\documentclass{amsart}
\usepackage{amssymb}
\usepackage{amsmath}
\usepackage{amsfonts}
\usepackage{mathrsfs}
\newtheorem{thm}{Theorem}[section]

\newtheorem{lem}[thm]{Lemma}

\theoremstyle{definition}

\theoremstyle{remark}
\newtheorem{rem}[thm]{Remark}

\numberwithin{equation}{section}
\usepackage{enumitem}
\usepackage{romannum}

\begin{document}
	
	%
	%
	%
	%
	%
	%
	%
	%
	%

	\title[BSE- property of $L^1(G,\omega, \mathcal A)$]{On the BSE- property of vector valued Beurling algebra $L^1(G,\omega, \mathcal{A})$}

\author[J. J. Dabhi]{Jekwin J. Dabhi}
\address{Institute of Infrastructure Technology Research and Management (IITRAM), Ahmedabad - 380026, Gujarat, India}
\email{jekwin.13@gmail.com, jekwin.dabhi@iitram.ac.in}
\thanks{The first author is thankful  for  research support from CSIR HRDG, India (file no. 09/1274(13914)/2022-EMR-I) for the Senior Research Fellowship. The second author is grateful to the National Board for Higher Mathematics (NBHM), India, for the research grant (02011/39/2025/NBHM(R. P.)/R \&D II/16090)}
	
\author[P. A. Dabhi]{Prakash A. Dabhi}
\address{Institute of Infrastructure Technology Research and Management (IITRAM), Ahmedabad - 380026, Gujarat, India}
\email{lightatinfinite@gmail.com, prakashdabhi@iitram.ac.in}
\thanks{}
	
	
	\date{}
	
\begin{abstract}
Let $G$ be a locally compact abelian group, and let $\omega:G \to [1,\infty)$ be a measurable weight, i.e., $\omega$ is measurable, and $\omega(s+t)\leq \omega(s)\omega(t)$ for all $s, t \in G$. Let $\mathcal{A}$ be a semisimple commutative Banach algebra with a predual $\mathcal A_\ast$ such that the Gel'fand space $\Phi_{\mathcal A}\subset \mathcal{A}_\ast$. If $\omega^{-1}$ is vanishing at infinity, then we show that the Banach algebra $L^1(G,\omega,\mathcal{A})$ is a BSE- algebra if and only if $\mathcal A$ is a BSE- algebra.
\end{abstract}

	\maketitle
	
\section{Introduction}

The study of multiplier algebras and their relationship with bounded continuous functions on the maximal ideal space has been a central theme in abstract harmonic analysis. To establish this framework formally, let $\mathcal A$ be a commutative Banach algebra. A nonzero linear map $\varphi:\mathcal A \to \mathbb C$ is a \emph{complex homomorphism} if $\varphi(ab)=\varphi(a)\varphi(b)$ for all $a, b \in \mathcal A$. Let $\Phi_{\mathcal A}$ be the collection of all complex homomorphisms on $\mathcal A$. For $a \in \mathcal A$, define $\widehat a:\Phi_{\mathcal A} \to \mathbb C$ by $\widehat a(\varphi)=\varphi(a)\;(\varphi \in \Phi_{\mathcal A})$. The smallest topology on $\Phi_{\mathcal A}$ such that each $\widehat a$, $a \in \mathcal A$, is continuous is the \emph{Gel'fand topology} on $\Phi_{\mathcal A}$, and $\Phi_{\mathcal A}$ with this topology is the \emph{Gel'fand space} of $\mathcal A$. The Gel'fand space $\Phi_{\mathcal A}$ is a locally compact Hausdorff space. Note that $\widehat a \in C_0(\Phi_{\mathcal A})$ for all $a \in \mathcal A$, where $C_0(\Phi_{\mathcal A})$ is the collection of all complex-valued continuous functions on $\Phi_{\mathcal A}$ vanishing at infinity. Let $\widehat{\mathcal A} = \{\widehat{a}: a \in \mathcal{A} \}$. The Banach algebra $\mathcal A$ is \emph{semisimple} if $\bigcap_{\varphi \in \Phi_{\mathcal A}}\ker \varphi=\{0\}$, where $\ker \varphi=\{a\in \mathcal A:\varphi(a)=0\}$.

Assume $\mathcal A$ is \emph{without order}, i.e., if $a \in \mathcal A$ and $a\mathcal A=\{0\}$, then $a=0$. Note that if $\mathcal A$ is semisimple, then it is without order. A map $T:\mathcal A \to \mathcal A$ is a \emph{multiplier} if $T(ab)=aTb=(Ta)b$ for all $a, b \in \mathcal A$. Let $M(\mathcal A)$ be the collection of all multipliers on $\mathcal A$. Then $M(\mathcal A)$ is a closed subalgebra of $B(\mathcal A)$, the Banach algebra of all bounded linear maps from $\mathcal A$ to $\mathcal A$ \cite[Proposition 1.4.11]{E}. The Banach algebra $M(\mathcal A)$ contains the identity $I(a)=a\;(a\in \mathcal A)$, and $\mathcal A$ is an ideal in $M(\mathcal A)$ via the identification $a \mapsto L_a$, where $L_a(b)=ab$ for all $b \in \mathcal A$. If $T \in M(\mathcal A)$, then, by \cite[Theorem 1.2.2]{L}, there is a unique bounded continuous function $\widehat T:\Phi_{\mathcal A}\to \mathbb C$ such that $\widehat{(Ta)} (\varphi)=\widehat T(\varphi)\widehat a(\varphi)$ for all $\varphi \in \Phi_{\mathcal A}$ and $a \in \mathcal A$. Let $\widehat {M(\mathcal A)}=\widehat M(\mathcal A)=\{\widehat T:T \in M(\mathcal A)\}$.

Let $C_b(\Phi_{\mathcal A})$ be the collection of all complex-valued bounded continuous functions on $\Phi_{\mathcal A}$. A function $\sigma\in C_b(\Phi_\mathcal{A})$ is a \emph{BSE-function} if for any finite choices of $c_1, c_2,\ldots, c_n \in \mathbb{C}$ and the same number of $\varphi_1,\varphi_2,\ldots,\varphi_n \in \Phi_\mathcal{A}$, the inequality
$$\left|\sum_{i=1}^n c_i \sigma(\varphi_i)\right|\leq C\left\|\sum_{i=1}^n c_i\varphi_i\right\|_{\mathcal A^\ast}$$
holds for some constant $C>0$. Let $C_{\text{BSE}}(\Phi_{\mathcal A})$ denote the collection of all BSE-functions. Then $C_{\text{BSE}}(\Phi_{\mathcal A})$ is a commutative Banach algebra with the norm $\|\sigma\|_{\text{BSE}}$, the infimum of all $C$ in the above inequality \cite[Lemma 1]{T}. We have $\widehat{\mathcal{A}} \subset C_{\text{BSE}}(\Phi_{\mathcal{A}})$. If $\mathcal{A}$ has a weak bounded approximate identity \cite{Jo}, then $\widehat{M(\mathcal{A})} \subset C_{\text{BSE}}(\Phi_{\mathcal{A}})$ \cite[Corollary 5]{T}. A commutative Banach algebra $\mathcal{A}$ is defined as a \emph{BSE-algebra} if $\widehat{M(\mathcal{A})} = C_{\text{BSE}}(\Phi_\mathcal{A})$ \cite{T}. The acronym BSE stands for \emph{Bochner-Schoenberg-Eberlein} and refers to the famous theorem proved by Bochner and Schoenberg \cite{B,S} for the additive group of real numbers that $L^1(\mathbb R)$ is a BSE-algebra, and by Eberlein \cite{E2} for any locally compact abelian group $G$ that $L^1(G)$ is a BSE-algebra. This concept was then generalized to arbitrary commutative Banach algebras by Takahasi and Hatori in \cite{T}.

To extend these concepts to group algebras, let $G$ be a locally compact abelian group with Haar measure $m$. Let
$$L^1(G) = \left\{f : G \rightarrow \mathbb{C}: f \text{ is measurable, }\|f\|_1=\int_{G} |f(s)| dm(s) < \infty \right\}.$$
For $f, g \in L^1(G)$, let
\begin{equation}\label{convolution}
(f\star g)(x)=\int_G f(y)g(x-y)dm(y)\quad(x \in G).
\end{equation}
By \cite[Theorem 1.1.6]{R}, $f\star g \in L^1(G)$ and $\|f\star g\|_1 \leq \|f\|_1\|g\|_1$. The \emph{Gel'fand space} of $L^1(G)$ is identified with $\widehat{G}$, the collection of all continuous homomorphisms from $G$ to $\mathbb{T}=\{z\in \mathbb C:|z|=1\}$ \cite[Theorem 2.7.2]{E}. The multiplier algebra of $L^1(G)$ is identified with $M(G)$, the Banach algebra of all regular complex Borel measures on $G$ \cite[Corollary 0.1.1]{L}.

A map $\omega:G \to [1,\infty)$ is a \emph{weight} \cite[Definition 1.6.1]{H} if $\omega$ is measurable, locally bounded, and $\omega(s+t)\leq \omega(s)\omega(t)$ for all $s,t \in G$. Let
$$L^1(G,\omega)=\left\{ f:G \to \mathbb C: f \text{ is measurable, }\|f\|_\omega=\int_G |f(s)|\omega(s)dm(s)<\infty \right\}.$$
Then $L^1(G,\omega)$ is a commutative Banach algebra, called a \emph{Beurling algebra}. Its multiplier algebra is isometrically isomorphic to $M(G,\omega)=\{\mu \in M(G):\|\mu\|_\omega =\int_G \omega d|\mu|<\infty\}$. The Gel'fand space $\Phi_{L^1(G,\omega)}$ is identified with $\widehat G_\omega$, the collection of all $\omega$-bounded characters on $G$.

To investigate vector-valued spaces, let $\mathcal{A}$ and $\mathcal{B}$ be commutative Banach algebras. The projective tensor product $\mathcal{A} \widehat{\otimes}_{\pi} \mathcal{B}$ is the completion of the algebraic tensor product $\mathcal{A} \otimes \mathcal{B}$ with respect to the projective norm $\|\cdot\|_{\pi}$. The injective tensor product $\mathcal{A} \widehat{\otimes}_{\epsilon} \mathcal{B}$ is the completion with respect to the injective norm $\|\cdot\|_\epsilon$. The Gel'fand space $\Phi_{\mathcal A \widehat{\otimes}_\pi \mathcal B}$ is identified with $\Phi_{\mathcal A}\times \Phi_{\mathcal B}$ \cite[Theorem 2.11.2]{E}.

For a unital commutative Banach algebra $\mathcal A$, let
$$L^1(G,\omega, \mathcal{A})=\left\{f:G \to \mathcal A: f \text{ is } m\text{-measurable,} \ \|f\|_\omega=\int_G \|f(s)\|\omega(s)dm(s)<\infty\right\}.$$
By \cite[Proposition 1.5.4]{E}, $L^{1}(G,\omega,\mathcal{A})$ is isomorphic to $L^1(G, \omega) \widehat{\otimes}_{\pi} \mathcal{A}$. The \emph{Gel'fand space} of $L^1(G,\omega,\mathcal{A})$ is identified with $\widehat{G}_\omega \times \Phi_{\mathcal{A}}$ \cite[Lemma 2.11.1]{E}. The multiplier algebra of $L^1(G,\omega,\mathcal{A})$ is identified with $M(G,\omega,\mathcal{A})$, the collection of all $\mathcal{A}$-valued regular measures on $G$ of bounded $\omega$-variation. A Banach space $\mathcal B$ is said to be a \emph{predual} of a Banach space $\mathcal A$ if the dual of $\mathcal B$ is $\mathcal A$. We denote a predual of $\mathcal A$ by $\mathcal A_\ast$.   If $\mathcal{A_\ast}$ is a predual of $\mathcal A$, let $C_0(G,\omega^{-1},\mathcal{A_\ast})$ be the set of all continuous functions $f:G \rightarrow \mathcal{A_\ast}$ such that $\frac{f}{\omega} \in C_0(G, \mathcal{A_\ast})$. Then $M(G,\omega,\mathcal{A})$ is the dual space of $C_0(G,\omega^{-1},\mathcal{A_\ast})$, and $C_0(G,\omega^{-1},\mathcal{A_\ast})$ is isomorphic to $C_0(G,\omega^{-1}) \widehat{\otimes}_\epsilon \mathcal{A_\ast}$ \cite[Proposition 1.5.6]{E}. (Note that setting $\omega \equiv 1$ yields the standard vector-valued group algebra $L^1(G,\mathcal A)$).

Recently, significant attention has shifted to the structural properties of these vector-valued function spaces. The authors in \cite{Re} claimed to have proved that for a unital semisimple commutative Banach algebra $\mathcal{A}$, the space $L^1(G,\mathcal{A})$ is a BSE-algebra if and only if $\mathcal{A}$ is a BSE-algebra. However, there is a fundamental gap in \textsection 3.1 of \cite{Re}. It is claimed there that a measure taking values in $(\text{span}(\Phi_{\mathcal A}))^\ast$ implies the measure is $C_{\text{BSE}}(\Phi_{\mathcal{A}})$-valued. This is not true in general. To demonstrate this, we recall that a topological Hausdorff space $X$ is \emph{stonean} if it is compact and the closure of every open set is open \cite[Definition 1.4.3]{DL}. Let $X$ be a non-empty, locally compact set, and let $\mu \in M(X)$. Then $\mu$ is \emph{normal} if $\langle f_\alpha, \mu \rangle=\int_X f_\alpha d\mu \rightarrow 0$ for each net $( f_\alpha)_{\alpha \in \Lambda}$ in $(C_{0}(X)^+, \leq)$ such that $f_{\alpha} \searrow 0$ in the lattice \cite[Definition 4.7.1]{DL}. Let $N(X)$ denote the collection of all normal measures on $X$. A non-empty, compact, Hausdorff space is \emph{hyperstonean} if it is stonean and $W_X = \bigcup\{\text{supp}(\mu) : \mu \in N(X)\}$ is dense in $X$ \cite[Definition 5.1.1]{DL}. By \cite[Theorem 6.4.1]{DL}, $C(X)$ has a predual if and only if $X$ is hyperstonean. Now, consider $\mathcal A = C(X)$, where $X$ is a compact, Hausdorff, but non-hyperstonean space. If $\text{span}(\Phi_\mathcal{A})^\ast$ were a subset of $C_{\text{BSE}}(\Phi_\mathcal A) = \mathcal A$, then $\mathcal A$ would have a predual (namely $\text{span}(\Phi_\mathcal{A})$), yielding a contradiction.

In this paper, we resolve this gap by introducing a precise structural requirement on the predual of the coefficient algebra, and we extend this transferability to the much broader class of weighted Beurling algebras $L^1(G, \omega, \mathcal{A})$. We proved in \cite{J} that if $\omega$ is a weight on $G$ such that $\omega^{-1}$ vanishes at infinity, then $L^1(G,\omega)$ is a BSE-algebra. Here, we demonstrate that imposing this asymptotic growth condition on the weight effectively neutralizes the topological obstructions of the group, allowing the BSE property to transfer perfectly to the vector-valued setting.

Our main results are as follows:
	
\begin{thm} \label{B1}
Let $\mathcal{A}$ be a semisimple commutative Banach algebra, let $\mathcal A$ have a predual $\mathcal A_\ast$ such that $\Phi_{\mathcal A} \subset \mathcal{A}_\ast$, and let $G$ be a locally compact abelian group. Then $L^1(G, \mathcal{A})$ is a BSE-algebra if and only if $\mathcal{A}$ is a BSE-algebra.
\end{thm}
	
\begin{thm}\label{B22}
Let $G$ be a locally compact abelian group, let $\omega$ be a weight on $G$ such that $\omega^{-1}$ vanishes at infinity, and let $\mathcal{A}$ be a semisimple commutative Banach algebra possessing a predual $\mathcal A_\ast$ such that $\Phi_{\mathcal A} \subset \mathcal{A}_\ast$. Then $L^1(G,\omega,\mathcal{A})$ is a BSE-algebra if and only if $\mathcal A$ is a BSE-algebra.
\end{thm}
	
\begin{rem}
If $\omega$ is a weight on a locally compact abelian group $G$, then, by \cite[Theorem 3.7.5]{H}, there is a continuous weight $\omega'$ on $G$ such that $L^1(G,\omega)$ and $L^1(G,\omega')$ are isomorphic as Banach algebras. This implies that $L^1(G,\omega,\mathcal A)$ and $L^1(G,\omega',\mathcal A)$ are isomorphic as Banach algebras for any commutative Banach algebra $\mathcal A$. Consequently, $L^1(G,\omega,\mathcal A)$ is a BSE-algebra if and only if $L^1(G,\omega',\mathcal A)$ is a BSE-algebra. Thus, we shall assume throughout the remainder of this paper that the weight $\omega$ is continuous.
\end{rem}

\section{$L^1(G,\mathcal{A})$ is a BSE- algebra.}
Let $G$ be a locally compact abelian group, and let $\mathcal{A}$ be a  unital  Banach algebra. Let $\overline{G}$ be the Bohr compactification of $G$. Then $G$ is a dense subset of $\overline{G}$ and also $G$ is continuously embedded in $\overline{G}$. For more details on Bohr compactification of a locally compact abelian group please see \cite[\textsection 1.8]{R}. 
Note that if $\mathcal{A}$ and $\mathcal{B}$ are  commutative Banach algebras, then $\mathcal{A}$ and $\mathcal{B}$ have bounded weak approximate identities if and only if $\mathcal{A} \widehat{\otimes}_\pi \mathcal{B}$ has a bounded weak approximate identity.
	
	\begin{lem}\label{2}
		Let $\mathcal{A}$ be a commutative Banach algebra without order, and let $\mathcal{A}$ have a predual $\mathcal{A}_\ast$ such that $\Phi_\mathcal A \subset \mathcal{A}_\ast$. Then the following are equivalent.
		\begin{enumerate}[label=(\roman*)]
			\item $\mathcal{A}$ has a weak bounded approximate identity.
			\item $\mathcal{A}$ is a BSE-algebra.
			\item $\mathcal{A}$ is unital.
		\end{enumerate}
	\end{lem}	
	\begin{proof} $(\mathrm{i}) \Rightarrow (\mathrm{ii})$
		 Let $\sigma \in C_{\text{BSE}}(\Phi_\mathcal{A})$. Then $\sigma$ extends to a continuous linear functional on $\text{span}(\Phi_\mathcal{A})$. As $\text{span}(\Phi_\mathcal{A}) \subset \mathcal{A}^\ast$, the Hahn-Banach extension theorem implies that  $\sigma$ extends as a continuous linear functional on $\mathcal{A^\ast}$. As $\mathcal A_\ast$ is canonically embedded into its second dual $\mathcal A^\ast$, $\sigma$ is also continuous linear function on $\mathcal A_\ast$.   Therefore $\sigma \in (\mathcal{A_\ast})^\ast = \mathcal{A}$,  that is $\sigma(x) = a(x)$ for all $x \in \mathcal A_\ast$ for some $a \in \mathcal{A}$. In particluar,  $\sigma(\varphi)=a(\varphi)=\widehat a(\varphi)$ for all $\varphi \in \Phi_{\mathcal A}$ as $\Phi_\mathcal A \subset \mathcal{A}_\ast$. Hence  $C_{\text{BSE}}(\Phi_\mathcal{A}) \subset \widehat{\mathcal{A}}$. Since $\mathcal A$ has a weak bounded approximate identity, we have  $\widehat{M(\mathcal{A})} \subset C_{\text{BSE}}(\Phi_\mathcal{A})$. Therefore is $\mathcal A$ a BSE- algebra.\\
		 Now, $(\mathrm{ii}) \Rightarrow (\mathrm{iii})$ is clear as $C_{\text{BSE}}(\Phi_\mathcal{A}) \subset \widehat{\mathcal{A}}$ by above argument. But $C_{\text{BSE}}(\Phi_\mathcal{A}) = \widehat{M(\mathcal{A})}$ implies $\widehat{M(\mathcal{A})} = \widehat{\mathcal{A}}$ that is $\mathcal A$ is unital.\\
		 $(\mathrm{iii}) \Rightarrow (\mathrm{i})$ is obvious.
	\end{proof}
	
	\begin{lem} \label{2.1} Let $G$ be a locally compact abelian group, and let $\mathcal A$ be a  commutative Banach algebra. For $\sigma \in C_{\text{BSE}}(\Phi_{L^1(G, \mathcal{A})})$ and $\psi \in \Phi_{\mathcal{A}}$, define $\sigma_{\psi} : \widehat{G} \rightarrow \mathbb{C}$ by $\sigma_{\psi}(\gamma) = \sigma(\varphi_{\gamma\psi})$ for all $\gamma \in \widehat G$. Then $\sigma_{\psi} \in C_{\text{BSE}}(\Phi_{L^1(G)})$.
	\end{lem}
	\begin{proof}
		By \cite[Theorem 2.11.2]{E}, $\Phi_{L^1(G,\mathcal{A})}=\widehat{G} \times \Phi_{\mathcal{A}}$. As $\sigma$ is continuous on  $\widehat{G} \times \Phi_{\mathcal{A}}$, $\sigma_{\psi}$ is continuous on $\widehat{G}$. For $\gamma_1, \gamma_2, \ldots, \gamma_n \in \widehat{G}$ and $c_1, c_2, \ldots, c_n \in \mathbb{C}$, we prove that $\|\sum_{i = 1}^{n} c_i\varphi_{\gamma_i\psi} \|_{L^1(G,\mathcal{A})^\ast} \leq \|\sigma\|_{\text{BSE}}\|\sum_{i = 1}^{n} c_i \gamma_i \|_{L^1(G)^\ast}$. Let $u \in L^1(G,\mathcal{A})$, and let $u= \sum_m f_m\otimes a_m$. Then
		\begin{eqnarray*}
			\left |\sum_{i = 1}^{n} c_i\varphi_{\gamma_i \psi}(u) \right | &=&  \left |\sum_{i = 1}^{n} c_i\varphi_{\gamma_i \psi}\left(\sum_m f_m\otimes a_m\right) \right | \\
			&=& \left |\sum_{i = 1}^{n} c_i \sum_m \widehat{f_m}(\gamma_i) \widehat{a_m}(\psi)\right | \\
			&=& \left |  \sum_m \left(\sum_{i = 1}^{n} c_i \gamma_i\right)(f_m)\, \widehat{a_{m}}(\psi)\right | \\
			&\leq&   \sum_m \|f_m\|\left\|\sum_{i = 1}^{n} c_i \gamma_i\right\|_{L^1(G)^\ast} \|a_{m}\|\|\psi\| \\
			&=& \left\|\sum_{i = 1}^{n} c_i \gamma_i\right\|_{L^1(G)^\ast} \|\psi\|\sum_m \|f_m\|\|a_m\|\\
			&\leq& \left\|\sum_{i = 1}^{n} c_i \gamma_i\right\|_{L^1(G)^\ast}\sum_m \|f_m\|\|a_m\|.
		\end{eqnarray*}
		This gives $\left\|\sum_{i = 1}^{n} c_i\varphi_{\gamma_i \psi}\right\|_{L^1(G,\mathcal A)^\ast}\leq \left\|\sum_{i = 1}^{n} c_i \gamma_i\right\|_{L^1(G)^\ast}$ . Since $\sigma$ is BSE, we have $\left|\sum_{i=1}^n c_i\sigma_{\psi}(\gamma_i)\right| \leq \|\sigma\|_{\text{BSE}}\|\sum_{i=1}^{n} c_i \gamma_i\|_{L^1(G)^\ast}$. Therefore $\sigma_{\psi} \in C_{\text{BSE}}(\Phi_{L^1(G)})$.
	\end{proof}
	\begin{proof}[Proof of Theorem \ref{B1}] Assume that $\mathcal A$ is a BSE- algebra. Then, by Lemma \ref{2}, $\mathcal A$ is a unital algebra. This implies that $L^1(G,\mathcal A)$ has a  bounded approximate identity. Therefore $\widehat M(G,\mathcal{A}) \subset C_{\text{BSE}}(\widehat{G}  \times \Phi_{\mathcal A})$.

Conversely, assume that $\sigma \in C_{\text{BSE}}(\widehat{G} \times \Phi_{\mathcal A})$. Then  $\sigma$ can be extended as a continuous linear functional on $\text{span}(\widehat{G} \times\Phi_{\mathcal{A}})$ in $L^1(G,\mathcal{A})^\ast$. As  $\Phi_{\mathcal A} \subset \mathcal A_{\ast}$, we have $\text{span}(\widehat{G} \times\Phi_{\mathcal{A}}) \subset \text{span}(\widehat{G} \times  \mathcal A_\ast)$ and $\sigma$ extends to a norm preserving functional on $\text{span}(\widehat{G} \times  \mathcal A_\ast)$. We shall denote this extension by $\sigma$ only.
		
		Let $\gamma \in \widehat{G}$ and let $x \in \mathcal{A_\ast}$. Define $F_{(\gamma, x)} : \overline{G} \rightarrow \mathcal A_\ast$ by $$F_{(\gamma, x)}(s) = \overline{\gamma(s)}x \quad(s \in \overline{G}).$$ Note that $F_{(\gamma, x)} \in C(\overline{G}, \mathcal A_\ast)$. Let $W = \text{span}\{F_{(\gamma,x)} : \gamma \in \widehat{G}, x \in \mathcal A_\ast \}$. Then $W\subset C(\overline G,\mathcal A_\ast)$. Define $T : W \rightarrow \mathbb{C}$ by $$T\left(\sum_{k=1}^n c_kF_{(\gamma_k, x_k)}\right)=\sum_{k=1}^n c_k \sigma(\varphi_{\gamma_k x_k}).$$
		Observe that $T$ is linear on $W$. Let $\sum_{k=1}^n c_kF_{(\gamma_k, x_k)} \in W$. Then
		\begin{eqnarray*}
			\left|T\left(\sum_{k=1}^n c_kF_{(\gamma_k, x_k)}\right)\right| & = & \left|\sum_{k=1}^n c_k \sigma(\varphi_{\gamma_k x_k}) \right|\\
			&\leq & \|\sigma\|_{\text{BSE}} \left\|\sum_{k=1}^n c_k \varphi_{\gamma_k x_k}\right\|_{L^1(G,\mathcal{A})^\ast}.
		\end{eqnarray*}
		Let $u \in L^1(G,\mathcal{A})$ with $\|u\|_{\pi} \leq 1$, and let $u=\sum_m \lambda_m f_m \otimes a_m$ with $\|f_m\| = \|a_m\|=1$ and $\sum_m |\lambda_m|  < \infty$. Then
		\begin{eqnarray*}
			\left|\sum_{k=1}^n c_k \varphi_{\gamma_k x_k}(u)\right | & = &  \left|\sum_{k=1}^n c_k \varphi_{\gamma_k x_k}\left(\sum_ m \lambda_m f_m \otimes a_m\right) \right|\\
			& = &  \left|\sum_{k=1}^n c_k \sum_m \lambda_m \widehat f_m(\gamma_k) a_m(x_k)\right|\\
			&=& \left| \sum_m \lambda_m  \sum_{k=1}^n c_k \widehat f_m(\gamma_k) a_m(x_k)\right|\\
			&\leq&  \sum_m \left| \lambda_m  \sum_{k=1}^n c_k \widehat f_m(\gamma_k) a_m(x_k)\right|\\
			&=&  \sum_m |\lambda_m|  \left| \sum_{k=1}^n c_k \widehat f_m(\gamma_k) a_m(x_k)\right|\\
			&\leq&  (\sum_m |\lambda_m|)  \sup \left \{\left | \sum_{k=1}^n c_k   \widehat{\mu}(\gamma_k) a(x_k) \right | :  \mu \in C(\overline{G})^{\ast}_1, \ a \in \mathcal{A}_1 \right\}.
		\end{eqnarray*}
		As $\|u\|_\pi\leq 1$, we have
		\begin{eqnarray*}
			\left|\sum_{k=1}^n c_k \varphi_{\gamma_k x_k}(u)\right| & \leq  & \sup \left \{\left | \sum_{k=1}^n c_k \widehat{\mu}(\gamma_k) a(x_k) \right | :  \mu \in C(\overline{G})^{\ast}_1, \ a \in \mathcal{A}_1 \right\}\\
			&=& \left\|\sum_{k=1}^n c_k F_{(\gamma_k, x_k)} \right\|_{\epsilon}.
		\end{eqnarray*}
		As $u \in L^1(G,\mathcal{A})$ is arbitrary, we have $\left \| \sum_{k=1}^n c_k \varphi_{\gamma_k  x_k} \right \|_{L^1(G,\mathcal{A})^\ast} \leq \left\| \sum_{k=1}^n c_k F_{(\gamma_k, x_k)} \right\|_{\epsilon}$. Now, using this in above inequality, we have $$\left | T\left(\sum_{k=1}^n c_kF_{(\gamma_k, x_k)}\right) \right | \leq \|\sigma\|_{\text{BSE}} \left\|\sum_{k=1}^n c_k F_{(\gamma_k, x_k)} \right\|_{\epsilon}.$$ Hence $T$ is a continuous linear functional on $W$. Therefore, by Hahn-Banach extension theorem, $T$ has a norm preserving extension to $C(\overline{G}, \mathcal{A_\ast})$. By Riesz representation theorem we get $\mu \in M(\overline{G},\mathcal{A})$ such that $$T(f)=\int_{\overline G}f(x)d\mu(x)\quad(f \in C(\overline G,\mathcal A_\ast)).$$ Let $\gamma \in \widehat G$ and $\psi \in \Phi_{\mathcal A}$. Then
		\begin{eqnarray*}
			\sigma(\varphi_{\gamma \psi})=T(F_{(\gamma,\psi)}) &= & \int_{\overline G}F_{(\gamma,\psi)}(s)d\mu(s)=\int_{\overline G}\overline{\gamma(s)}\psi d\mu(s)\\
			&=&\psi\left(\int_{\overline G}\overline{\gamma(s)} d\mu(s)\right)=\widehat \mu(\varphi_{\gamma \psi}).
		\end{eqnarray*}
		As $\overline{G}$ is compact, the function $F_{(\gamma,\psi)}(s) = \overline{\gamma(s)}\psi\;(s \in \overline G)$ is  uniformly continuous. Let $(E_{i})_{i \in I}$ be a partition of $\overline{G}$, where $I$ is a finite set. Then $\sum_{i} \gamma(s_i)\psi \mu(E_i)$ approximates the integral $\int_{\overline G}\overline{\gamma(s)}\psi d\mu(s)$. Now, $\sum_{i} \gamma(s_i)\psi \mu(E_i) = \sum_{i} \gamma(s_i)\psi (\mu(E_i))$ and as $\sum_{i} \gamma(s_i)\psi (\mu(E_i))$  approximates the integral $\int_{\overline G}\overline{\gamma(s)} d(\psi \circ \mu)(s)$, we have  $$\int_{\overline G}\overline{\gamma(s)}\psi d\mu(s) = \int_{\overline G}\overline{\gamma(s)} d(\psi \circ \mu)(s).$$
		Now, we show that $\mu$ is concentrated on $G$ and to show this we prove that  $\psi \circ \mu$ is concentrated on $G$ for all $\psi \in \Phi_{\mathcal A}$.
		Fix $\psi \in \Phi_\mathcal{A}$. Then we have $\sigma_\psi(\gamma)=\sigma(\varphi_{\gamma\psi}) = \int_{\overline{G}}\overline{\gamma(s)} \ d(\psi \circ \mu)(s)$. By  Lemma \ref{2.1}, $\sigma_\psi \in C_{\text{BSE}}(\Phi_{L^1(G)})$. By \cite[Theorem 1.9.1]{R}, the measure $|\psi \circ \mu|$ is concentrated on $G$, and it means that $\psi \circ \mu$ is concentrated at $G$. Let $E$ be a measurable subset of $G$. As $\psi \circ \mu$ is concentrated on $G$, we have $\psi(\mu(E)) = \psi(\mu(E \cap G))$ for all $\psi \in \Phi_\mathcal A$. As $\mathcal A$ is semisimple, $\mu(E) = \mu(E \cap G)$, i.e., $\mu$ is concentrated on $G$. Therefore, we have
		$$\sigma(\varphi_{\gamma\varphi}) =  \int_{G}\overline{\gamma(s)} \ d (\psi\circ \mu)(s)=\widehat \mu(\varphi_{\gamma\psi}).$$  Therefore, $C_{BSE}(\Phi_{\mathcal A}) \subset \widehat{M(G, \mathcal A)}$. Hence $L^1(G, \mathcal A)$ is a BSE- algebra. Conversely, assume that $L^1(G, \mathcal A)$ is a BSE- algebra. Then $L^1(G, \mathcal A)$ has a weak bounded approximate identity which implies that $\mathcal A$ has  a weak bounded approximate identity. Therefore, by Lemma \ref{2} we have $\mathcal A$ is a BSE- algebra.
	\end{proof}
	
\section{$L^1(G,\omega, \mathcal{A})$ is  BSE}
We now prove the Theorem \ref{B22}. We shall require some results in proving it.
\begin{lem}
	Let $\omega$ be a weight on a locally compact abelian group $G$, $\omega^{-1}$ vanish at infinity, and let $\mathcal A$ be a commutative Banach algebra with a predual $\mathcal A_\ast$. If $\gamma \in \widehat G_\omega$ and $x \in \mathcal{A}_\ast$, then $\varphi_{\gamma x}$ is a continuous linear functional on $L^1(G,\omega,\mathcal A)$, where $$\varphi_{\gamma x}\left(\sum_n f_n \otimes a_n\right)=\sum_n \widehat f_n(\gamma)a_n(x)\quad\left(\sum_n f_n \otimes a_n \in L^1(G,\omega,\mathcal A)\right).$$
\end{lem}
\begin{proof}
	Clearly, $\varphi_{\gamma x}$ is linear. Let $u \in L^1(G,\omega, \mathcal A)$ and let $u=\sum_n f_n \otimes a_n$ with $\sum_n \|f_n\|_\omega \|a_n\| < \infty$. Then
	\begin{eqnarray*}
		|\varphi_{\gamma x}(u)| & = &  \left|\varphi_{\gamma x}\left(\sum_n f_n \otimes a_n\right)\right|  =  \left|\sum_n \widehat f_n (\gamma)a_n(x)\right|\\
		& \leq & \sum_n |\widehat f_n (\gamma)| |a_n(x)|\\
		& \leq & \sum_n \|f_n\|_\omega \|a_n\| \|x\|.
	\end{eqnarray*}
	This gives $|\varphi_{\gamma x}(u)|\leq \|x\| \|u\|$. This proves the result.
\end{proof}

\begin{lem}
	Let $G$ be a locally compact abelian group, and let $\mathcal A$ be a  commutative Banach algebra with a predual $\mathcal A_\ast$. If $\epsilon >0$, then there is a neighborhood $U$ of $0$ in $\widehat G$ such that $\|\varphi_{\gamma_1 x}-\varphi_{\gamma_2 x}\|_{L^1(G,\mathcal A)^\ast}<\epsilon$ whenever $\gamma_1, \gamma_2 \in U$ and $x \in (\mathcal A_\ast)_1$.
\end{lem}
\begin{proof} Let $\epsilon > 0$. Then, by \cite[Theorem 2.2]{J}, there exists a neighborhood $U$ of $0$ such that $\| \varphi_{\gamma_1}-\varphi_{\gamma_2}\|_{L^1(G)^\ast}<\epsilon$ whenever $\gamma_1, \gamma_2 \in U$. Let $u=\sum_n f_n \otimes a_n \in L^1(G, \omega, \mathcal A)$, $x \in (\mathcal A_\ast)_1$, and let $\gamma_1, \gamma_2 \in U$. Note that $\sum_n \|f_n\||a_n(x)|\leq \sum_n \|f_n\|\|a_n\|$. We have
	\begin{eqnarray*}
		|\varphi_{\gamma_1 x}(u)-\varphi_{\gamma_2 x}(u)|&=& \left|\varphi_{\gamma_1 x}\left(\sum_n f_n\otimes a_n\right)-\varphi_{\gamma_2 x}\left(\sum_n f_n\otimes a_n\right)\right| \\
		&=& \left|\varphi_{\gamma_1}\left(\sum_n f_n a_n(x) \right)-\varphi_{\gamma_2}\left(\sum_n f_n  a_n(x)\right)\right|\\
		& \leq & \|\varphi_{\gamma_1}-\varphi_{\gamma_2}\|_{L^1(G)^\ast}\sum_n \|f_n\|\|a_n\|\\
		& \leq & \epsilon \sum_n \|f_n\|\|a_n\|.
	\end{eqnarray*}
	This gives $\|\varphi_{\gamma_1 x}-\varphi_{\gamma_2 x}\|_{L^1(G,\mathcal A)^\ast}\leq \epsilon$ whenever $\gamma_1, \gamma_2 \in U$ and $x\in (\mathcal A_\ast)_1$.
\end{proof}
\begin{lem}\label{l4}
	Let $G$ be a locally compact abelian group, and let $\mathcal A$ be a commutative Banach algebra with a predual $\mathcal A_\ast$. If $c_1, c_2,\ldots, c_n \in \mathbb{C}$, $f_1, f_2,\ldots, f_n \in L^1(\widehat G)$, $\sigma \in C_{\text{BSE}}(\Phi_{L^1(G, \omega, \mathcal{A})})$ and $x_1, x_2,\ldots,x_n \in (\mathcal{A}_\ast)_{1}$, then $$\left|\sum_{k=1}^n c_k \int_{\widehat{G}} f(\gamma) \sigma(\varphi_{\gamma^{-1} x_k}) d\gamma \right | \leq \|\sigma\|_{\text{BSE}} \left \| \sum_{k=1}^n c_k \int_{\widehat{G}} f(\gamma) \varphi_{\gamma^{-1} x_k} d\gamma  \right \|_{L^1(G,\omega,\mathcal{A})^\ast} .$$
\end{lem}
\begin{proof} Let $g_1 \ \text{and} \ g_2$ be in $C_c(\widehat G)$ with supports contained in compact sets $K_1$ and $K_2$ respectively. Let $(A_i)_{i=1}^{n}$ and $(B_j)_{j=1}^{k}$ be partitions of $K_1$ and $K_2$ respectively.  Then
	\begin{eqnarray*}
		&& \left| c_1 \sum_{i=1}^{n} g_1(\gamma_i)\sigma(\varphi_{{\gamma}^{-1}_i x_1})m(A_i) + c_2 \sum_{j=1}^{k} g_2(\gamma_j)\sigma(\varphi_{{\gamma}^{-1}_j x_2})m(B_j)  \right|\\
		&\leq & \|\sigma\|_{\text{BSE}} \left \| c_1 \sum_{i=1}^{n} g_1(\gamma_i)m(A_i)\varphi_{{\gamma}^{-1}_i x_1} + c_2 \sum_{j=1}^{k} g_2(\gamma_j)m(B_j)\varphi_{{\gamma}^{-1}_j x_2}\right\|_{L^1(G,\omega,\mathcal{A})^\ast}.
	\end{eqnarray*}
	Now, applying limit over such partitions of $K_1$ and $K_2$ we get
	\begin{eqnarray*}
		&&\left|c_1 \int_{\widehat{G}} g_1(\gamma) \sigma(\varphi_{{\gamma}^{-1} x_1}) d\gamma+ c_2 \int_{\widehat{G}}  g_2(\gamma) \sigma(\varphi_{{\gamma}^{-1} x_2}) d\gamma \right|\\
		& \leq & \|\sigma\|_{\text{BSE}} \left \| c_1 \int_{\widehat{G}}   g_1(\gamma)\varphi_{{\gamma}^{-1} x_1} d\gamma + c_2 \int_{\widehat{G}}   g_2(\gamma)\varphi_{{\gamma}^{-1} x_1} d\gamma \right \|_{L^1(G,\omega,\mathcal{A})^\ast}.
	\end{eqnarray*}
	Applying the same process to $g_1, g_2,\ldots,g_n \in  C_c(\widehat{G})$ and $x_1,\ldots,x_n \in (\mathcal A_\ast)_1$, we get
	\begin{eqnarray*}
		\left|\sum_{i=1}^nc_i \int_{\widehat{G}} g_i(\gamma) \sigma(\varphi_{{\gamma}^{-1} x_i}) d\gamma\right| \leq \|\sigma\|_{\text{BSE}} \left \|\sum_{i=1}^n c_i \int_{\widehat{G}}   g_i(\gamma)\varphi_{{\gamma}^{-1} x_i} d\gamma\right\|_{L^1(G,\omega,\mathcal{A})^\ast}.
	\end{eqnarray*}
	Take any $\epsilon >0$. Since $C_c(\widehat G)$ is dense in $L^1(\widehat G)$, we get $g_1,\ldots,g_n \in C_c(\widehat{G})$ such that $\int_{\widehat{G}}|f_i(\gamma) - g_i(\gamma)| d\gamma < \frac{\epsilon}{n}$ for $1\leq i \leq n$. Let $M>0$ be such that $|\sigma(\varphi_{\gamma x})| \leq M$ for all $\gamma \in \widehat{G}$ and $x \in (\mathcal{A}_{\ast})_1$. So,
	$$\left | \sum_{i=1}^{n} \int_{\widehat{G}}   g_i(\gamma)\sigma(\varphi_{{\gamma}^{-1} x_i}) d\gamma - \sum_{i=1}^{n} \int_{\widehat{G}}   f_i(\gamma)\sigma(\varphi_{{\gamma}^{-1} x_i}) d\gamma \right | \leq M\epsilon$$
	and hence
	$$\left | \sum_{i=1}^{n} \int_{\widehat{G}}   f_i(\gamma)\sigma(\varphi_{{\gamma}^{-1} x_i}) d\gamma \right | \leq  \left | \sum_{i=1}^{n} \int_{\widehat{G}}   g_i(\gamma)\sigma(\varphi_{{\gamma}^{-1} x_i}) d\gamma \right |  + M\epsilon.$$
	Now, \begin{eqnarray*}
		\left \| \sum_{i=1}^{n} \int_{\widehat{G}}   (g_i(\gamma)-f_i(\gamma))\varphi_{{\gamma}^{-1} x_i} d\gamma\right \|_{L^1(G,\omega,\mathcal{A})^\ast} & \leq & \sum_{i=1}^{n} \int_{\widehat{G}}|g_i(\gamma) - f_i(\gamma)|\|\varphi_{{\gamma}^{-1} x_i}\|_{L^1(G,\omega,\mathcal{A})^\ast} d\gamma\\
		& < & \epsilon
	\end{eqnarray*}
	and hence
	\begin{eqnarray*}
		\left \| \sum_{i=1}^{n} \int_{\widehat{G}}  g_i(\gamma)\varphi_{{\gamma}^{-1} x_i} d\gamma \right \|_{L^1(G,\omega,\mathcal{A})^\ast} \leq \epsilon + \left \| \sum_{i=1}^{n} \int_{\widehat{G}}   f_i(\gamma)\varphi_{{\gamma}^{-1} x_i)} d\gamma \right \|_{L^1(G,\omega,\mathcal{A})^\ast}.
	\end{eqnarray*}
	Finally,
	\begin{eqnarray*}
		\left | \sum_{i=1}^{n} \int_{\widehat{G}}   f_i(\gamma)\sigma(\varphi_{{\gamma}^{-1} x_i}) d\gamma \right | & < & M\epsilon + \left |  \sum_{i=1}^{n} \int_{\widehat{G}} g_i(\gamma)\sigma(\varphi_{{\gamma}^{-1} x_i}) d\gamma  \right | \\
		& \leq & M\epsilon + \|\sigma\|_{\text{BSE}} \left \|  \sum_{i=1}^{n} \int_{\widehat{G}}   g_i(\gamma)\varphi_{{\gamma}^{-1} x_i} d\gamma  \right \|_{L^1(G,\omega,\mathcal{A})^\ast} \\
		& \leq & M\epsilon + \epsilon\|\sigma\|_{\text{BSE}} \\
		&& + \|\sigma\|_{\text{BSE}} \left \|  \sum_{i=1}^{n} \int_{\widehat{G}}   f_i(\gamma)\varphi_{{\gamma}^{-1} x_i} d\gamma \right \|_{L^1(G,\omega,\mathcal{A})^\ast}.
	\end{eqnarray*}
	This proves the result.
\end{proof}

\begin{lem}\cite[Lemma 2.4]{J}
	Let $G$ be a locally compact abelian group, and let $$\mathscr{F} = \left\{f \in  C_{c}(G) :\widehat f \in L^1(\widehat G),\, f(s) = \int_{\widehat G}\widehat{f}(\gamma)\gamma(s)d\gamma\;(s \in G)\right\}.$$  Then $ \mathscr{F}$ is a dense subspace in $C_{0}(G)$.
\end{lem}

\begin{lem} \label{Dense}  If $\omega$ is a weight on a locally compact abelian group $G$, and $\mathcal A$ is a commutative Banach algebra with a predual $\mathcal A_\ast$, then $$\mathscr{F}_{\mathcal A} = \left\{\sum\limits_{k=1}^{n} f_k \otimes x_k :f_k \in \mathscr{F}, x_k \in \mathcal{A_\ast}, n \in \mathbb{N} \right\}$$ is a dense subspace of $C_0(G, {\omega}^{-1}, \mathcal{A_\ast})$.
\end{lem}
\begin{proof}
	Clearly, $\mathscr{F}_{\mathcal A}$ is a subspace of $C_0(G, {\omega}^{-1}, \mathcal{A_\ast})$.  Let $u =  \sum_{k=1}^n f_k \otimes x_k \in C_0(G, \omega^{-1}, \mathcal{A_\ast})$, and let $\epsilon >0$. Since $\mathscr{F}$ is dense in $C_0(G,{\omega}^{-1})$, there exist $g_1, g_2,\ldots, g_n \in \mathscr{F}$ such that $\|f_k - g_k\|_{\omega^{-1}} < \frac{\epsilon}{n(1+ \|x_k\|)}$ for all $1\leq k \leq n$.
	Now, let $\varphi \in (C_0(G))_{1}^\ast$ and let $\psi \in (\mathcal A_\ast)_{1}^\ast=\mathcal A_1$. Then
	\begin{eqnarray*}
		\left | (\varphi,\psi)\left(\sum_{k=1}^n f_k \otimes x_k\right) -  (\varphi,\psi)\left(\sum_{k=1}^n g_k \otimes x_k\right) \right | & = & \left | \sum_{k=1}^n \varphi(f_k - g_k)  \psi(x_k)  \right | \\
		& \leq &  \sum_{k=1}^n |\varphi(f_k - g_k)  \psi(x_k)  | \\
		& \leq & \sum_{k=1}^n \|f_k - g_k\|_{\omega^{-1}}  \|x_k\| \\
		& < & \sum_{k=1}^n \frac{\epsilon}{n(1 + \|x_k\|)} \|x_k\| \\
		& < & \epsilon.
	\end{eqnarray*}
	Now, applying supremum over such $(\varphi, \psi)$ we get the required result.
\end{proof}

We shall denote the projective norm on $L^1(G,\omega) \widehat{\otimes}_\pi \mathcal A$ by $\|\cdot\|_{\pi \omega}$, projective norm on $L^1(G)\widehat{\otimes}_\pi \mathcal A$ by $\|\cdot \|_\pi$, the injective norm on $C_0(G,\omega^{-1},\mathcal A_\ast)$ by $\|\cdot\|_{\epsilon \omega^{-1}}$ and the injective norm on $C_0(G,\mathcal A_\ast)$ by $\|\cdot\|_\epsilon$.

\begin{proof}[Proof of Theorem \ref{B22}]
Assume that $\mathcal A$ is a BSE-algebra. Then by Lemma \ref{2} $\mathcal A$ is unital. Therefore $L^1(G, \omega, \mathcal A)$ has a bounded approximate identity. Hence, by \cite[Corollary 5]{T}, $\widehat{M}(L^1(G,\omega,\mathcal{A})) \subset  C_{\text{BSE}} (\Phi_{L^1(G,\omega,\mathcal{A}})$.   Now, let $\sigma \in C_{\text{BSE}} (\Phi_{L^1(G,\omega,\mathcal{A}})$. Then for all  $n \in \mathbb N$, $c_1,c_2,\ldots,c_n \in \mathbb C$, $\gamma_1,\gamma_2,\ldots,\gamma_n \in \Phi_{L^1(G,\omega)}$, and $\psi_1, \psi_2,\ldots,\psi_n \in \Phi_{\mathcal{A}}$, we have
\begin{eqnarray}\label{BSE}
\left|\sum_{i=1}^n c_i\sigma(\varphi_{\gamma_i \psi_i})\right|\leq \|\sigma\|_{\text{BSE}}\left\|\sum_{i=1}^nc_i\varphi_{\gamma_i\psi_i}\right\|_{L^1(G,\omega, \mathcal{A})^\ast}.
\end{eqnarray}
	Note that if $f \in L^1(G,\omega)$ and $\|f\|_\omega \leq 1$, then $\|f\|_1\leq 1$. This implies that $\{u\in L^1(G,\omega)\otimes_{\pi}\mathcal A:\|u\|_{\pi ,\omega}\leq 1\} \subset \{u\in L^1(G) \widehat{\otimes}_{\pi}\mathcal A:\|u\|_\pi \leq 1\}$. So, if $\Phi \in L^1(G, \mathcal{A})^\ast \cap L^1(G,\omega, \mathcal{A})^\ast$, then $\|\Phi\|_{L^1(G,\omega, \mathcal{A})^\ast}=\sup\{|\Phi(u)|: u \in L^1(G,\omega, \mathcal{A}), \|u\|_{\pi \omega} \leq 1\}\leq \sup\{|\Phi(u)|:u \in L^1(G,\mathcal{A}), \|u\|_\pi\leq 1\}=\|\Phi\|_{L^1(G,\mathcal{A})^\ast}$. Let $\gamma_1,\gamma_2,\ldots,\gamma_n \in \widehat G$, $\psi_1,\psi_2,\ldots,\psi_n \in \Phi_{\mathcal{A}}$ and $c_1,c_2,\ldots,c_n \in \mathbb C$. Observe that $\sum_{i=1}^nc_i\varphi_{\gamma_i\psi_i} \in L^1(G,\mathcal{A})^\ast \cap L^1(G,\omega,\mathcal{A})^\ast$. So, by inequality (\ref{BSE}), we have
	\begin{eqnarray*}
		\left|\sum_{i=1}^n c_i\sigma(\varphi_{\gamma_i\psi_i})\right| \leq  \|\sigma\|_{\text{BSE}}\left\|\sum_{i=1}^nc_i \varphi_{\gamma_i\psi_i}\right\|_{L^1(G,\omega, \mathcal{A})^\ast} \leq \|\sigma\|_{\text{BSE}}\left\|\sum_{i=1}^nc_i \varphi_{\gamma_i\psi_i}\right\|_{{L^1(G, \mathcal A)}^\ast}.
	\end{eqnarray*}
	It means that $\sigma$ is a BSE- function  on $\Phi_{L^1(G, \mathcal{A})}$. Since $L^1(G,\mathcal{A})$ is a BSE- algebra, by Theorem \ref{B1},  there is $u \in M(G,\mathcal{A})$ such that
	\begin{equation*}
		\sigma(\varphi_{\gamma \psi}) =  \widehat u(\varphi_{\gamma \psi}) \quad(\gamma \in \widehat G, \psi \in \Phi_{\mathcal A}).
	\end{equation*}
We prove that $u \in M(G, \omega, \mathcal{A})$, the dual space of $C_0(G,\omega^{-1},\mathcal{A_\ast})$. We show that the map $\sum_{k=1}^n f_k \otimes x_k \mapsto \sum_{k=1}^n \int_{G} f_k(s)  d (x_k\circ \mu)(s) $ from $\mathscr F_{\mathcal A}$ to $\mathbb C$ is a bounded linear map, where $\mathscr{F}_{\mathcal A}$ is as in Lemma \ref{Dense}. Let us call this map $\eta$, i.e.,
\begin{eqnarray*}
\eta\left(\sum_{k=1}^n f_k \otimes x_k\right) = \sum_{k=1}^n \int_{G} f_k(s)  d (x_k\circ \mu)(s).
\end{eqnarray*}
The above integral makes sense as each $f_k$ is bounded and $ x_k \circ \mu  \in M(G)$ for all $1\leq k \leq n$. Observe that $\eta$ is linear on $\mathscr F_\mathcal A$. Let $\sum_{k=1}^n f_k \otimes x_k \in \mathscr F_{\mathcal A}$. Then
	\begin{eqnarray}\label{e}
		\left | \eta\left(\sum_{k=1}^n f_k \otimes x_k\right) \right | &=& \left |\sum_{k=1}^n \int_{G} f_k(s)d(x_k\circ \mu)(s)\right | \nonumber\\
		& = & \left | \sum_{k=1}^n \int_{G} \left (\int_{\widehat{G}} \widehat{f}_k(\gamma) \gamma (s) d\gamma \right ) d(x_k\circ \mu)(s)\right | \nonumber\\
		&=& \left | \sum_{k=1}^n \int_{\widehat G} \widehat{f}_k(\gamma) \left (\int_{G}  \gamma (s) d(x_k\circ \mu)(s) \right )d\gamma   \right | \nonumber\\
		&=& \left | \sum_{k=1}^n \int_{\widehat G} \widehat{f}_k(\gamma)  \left (\int_{G}  \gamma (s)x_k d\mu(s) \right )d\gamma  \right | \quad[\because \text{Fubini's theorem}] \nonumber \\
		&=& \left | \sum_{k=1}^n \int_{\widehat{G}} \widehat{f}_k(\gamma) \sigma(\varphi_{{\gamma}^{-1} x_k})d\gamma \right |  \nonumber\\
		&\leq& \|\sigma\|_{\text{BSE}} \left \| \sum_{k=1}^n \int_{\widehat{G}}  {\widehat{f}_k(\gamma)} \varphi_{{\gamma}^{-1} x_k}d\gamma \right \|_{L^1(G, \omega ,\mathcal{A})^\ast}
	\end{eqnarray}
The last inequality holds because of Lemma $\ref{l4}$. For $g \in L^1(G,\omega)$, define $\Theta_g(f)=\int_G f(s)g(s)ds\;(f \in C_0(G,\omega^{-1}))$. Then $|\Theta_g(f)|=\left|\int_G \frac{f(s)}{\omega(s)}g(s)\omega(s)ds\right|\leq \|f\|_{\infty\omega^{-1}}\|g\|_\omega$ for all $f \in C_0(G,\omega^{-1})$. So, $\Theta_g$ is a continuous linear functional on $C_0(G,\omega^{-1})$ having norm at most $\|g\|_\omega$. Let $g = \sum_m  g_m \otimes a_m \in L^1(G,\omega,\mathcal{A})$. Then
	\begin{eqnarray*}
		\left|\sum_{k=1}^n\left(\int_{\widehat G} {\widehat f_k(\gamma)}\varphi_{{\gamma}^{-1} x_k} d\gamma \right)(g)\right|& = & \left|\sum_{k=1}^n\int_{\widehat G} {\widehat f_k(\gamma)}\varphi_{{\gamma}^{-1} x_k}(g)d\gamma \right| \\
		&=&\left|\sum_{k=1}^n\int_{\widehat G} {\widehat f_k(\gamma)} g(\varphi_{{\gamma}^{-1} x_k})d\gamma\right|\\
		& = & \left|\sum_{k=1}^n\int_{\widehat G} {\widehat f_k(\gamma)}\left(\sum_m \widehat{g_m}({\gamma}^{-1}) a_m(x_k)\right)d\gamma\right|\\
		& = & \left|\sum_{k=1}^n\int_{\widehat G} {\widehat f_k(\gamma)}\left(\sum_m a_m(x_k)\int_{G} g_m(s) {\gamma(s)} ds \right) d\gamma\right|\\
		& = & \left| \sum_m \int_{G}g_m(s)\sum_{k=1}^n\left(\int_{\widehat G} {\widehat f_k(\gamma)} {\gamma(s)}d\gamma\right) a_m(x_k)ds\right|\\
		& = & \left|\sum_m \int_{G}g_m(s) \sum_{k=1}^n {f_k(s)} a_m(x_k) ds\right|\\
		& \leq & \sum_m \left|\sum_{k=1}^n a_m(x_k)\int_{G}g_m(s) {f_k(s)}  ds\right|\\
		& = & \sum_m \left|\sum_{k=1}^n \Theta_{g_m}(f_k) a_m(x_k) ds\right| \\
		& \leq & \sum_m  \|g_m\|_{\omega} \|a_m\| \left\|\sum_{k=1}^n  f_k \otimes x_k\right\|_{\epsilon, \omega^{-1}} \\
		& \leq & \|g\|_{\pi \omega} \|\sum_{k=1}^n  f_k \otimes a_k\|_{\epsilon, \omega^{-1}}.
	\end{eqnarray*}
	This proves that $$\left \|\sum_{k=1}^n \int_{\widehat{G}} \widehat{f}_k(\gamma) \varphi_{({\gamma}^{-1},x_k)}d\gamma \right\|_{L^1(G, \omega ,\mathcal{A})^\ast}\leq \|\sum_{k=1}^n  f_k \otimes a_k\|_{\epsilon, \omega^{-1}}$$ and hence, by the inequality (\ref{e}),  $$\left|\eta\left(\sum_{k=1}^n f_k \otimes x_k\right)\right |\leq \|\sigma\|_{\text{BSE}}\|\sum_{k=1}^n f_k \otimes a_k\|_{\epsilon, \omega^{-1}}$$ for all  $\sum_{k=1}^n f_k \otimes x_k \in \mathscr F_\mathcal A$. As $\mathscr F_\mathcal A$ is dense in $(C_0(G,\omega^{-1},\mathcal{A}_\ast),\|\cdot\|_{\epsilon, \omega^{-1}})$, $\eta$ has a unique norm preserving extension on $C_0(G,\omega^{-1},\mathcal{A}_\ast)$ and this extension will be $$\eta(f)=\int_G f(s)du(s)\quad(f \in C_0(G,\omega^{-1},\mathcal{A}_\ast).$$
	This implies that $ u\in M(G,\omega,\mathcal{A})$. Hence $L^1(G, \omega, \mathcal A)$ is a BSE- algebra. Conversely, assume that $L^1(G, \omega, \mathcal A)$ is a BSE- algebra. Then $\mathcal A$ has a weak bounded approximate identity. Therefore, by Lemma \ref{2} $\mathcal A$ BSE- algebra. This completes the proof.

\end{proof}

\section{Examples}
In this section we give some examples of unital commutative Banach algebras $\mathcal A$ having predual $\mathcal A_\ast$ and the Gel'fand space $\Phi_{\mathcal A}$ is contained in the predual $\mathcal A_\ast$.

\begin{enumerate}
	\item Let $X$ be hyperstonean. Then, by \cite[Theorem. 6.4.1]{DL},  $C(X)$ has the predual namely $N(X)$, the collection of all regular normal measures on $X$. Clearly, $\Phi_{C(X)}= X$ is embedded in $N(X)$ via mapping $x \mapsto \delta_x$, where $\delta_x$ is the Dirac measure concentrated at $x$. Clearly, $\Phi_{C(X)}= X$ is embedded in $N(X)$ via mapping $x \mapsto \delta_x$, where $\delta_x$ is the Dirac measure concentrated at $x$.
	
	\item Consider the semigroup $(\mathbb N_0,+)$. Let $\omega$ be a weight on $\mathbb N_0$ with $\omega(n)\to \infty$ as $n \to \infty$ and $\lim_{n\to \infty}\omega(n)^{\frac{1}{n}}=1$. Then $\ell^1(\mathbb N_0,\omega)$ is a unital, semisimple commutative Banach algebra, $\Phi_{\ell^1(\mathbb N_0,\omega)}=\{\varphi_z:|z|\leq 1\}$. The complex homomorphism $\varphi_z$ corresponds to the sequence $(1,z,z^2,\ldots)$. Note that $c_0(\mathbb N_0,\omega^{-1})^\ast=\ell^1(\mathbb N_0,\omega)$ and $\varphi_z \in c_0(\mathbb N_0,\omega^{-1})$.
	
	\item Consider $\mathbb N$ with the multiplication binary operation. Let $\omega$ be a weight on $\mathbb N$ such that $\omega(n)\to \infty$ as $n\to \infty$ and $\lim_{n\to \infty}\omega(k^n)^{\frac{1}{n}}=1$ for all $k \in \mathbb N$. Then $\ell^1(\mathbb N,\omega)$ is a unital Banach algebra with convolution multiplication. The Gel'fand space $\Phi_{\ell^1(\mathbb N,\omega)}$ can be identified with $\overline {\mathbb D}^\infty=\overline{\mathbb D}\times \overline{\mathbb D} \times \overline{\mathbb D} \times \cdots$, where $\overline{\mathbb D}=\{z\in \mathbb C:|z|\leq 1\}$ \cite{GL}. A map $\varphi$ is a complex homomorphism on $\ell^1(\mathbb N,\omega)$ if and only if there exists a semicharacter $\chi$ on $\mathbb N$ such that $\chi(1)=1$, $\chi(p_i)=z_i$ for all $i$, $|z_i|\leq 1$ and $\varphi=\varphi_\chi$, where $\{p_1,p_2,\ldots\}$ is the sequence of all primes written in increasing order and $\varphi_\chi(f)=\sum_{n=1}^\infty f(n)\chi(n)$ for all $f \in \ell^1(\mathbb N,\omega)$. If $n=p_{i_1}^{\alpha_1}\cdots p_{i_k}^{\alpha_k}$, then $\chi(n)=z_{i_1}^{\alpha_1}\cdots z_{i_k}^{\alpha_k}$. So, $\varphi_\chi$ can be seen as a sequence whose $n$-th term is $z_{i_1}^{\alpha_1}\cdots z_{i_k}^{\alpha_k}$. Since $|z_{i_j}|\leq 1$ for all $j$, we have $\frac{|\chi(n)|}{\omega(n)}\to 0$ as $n \to \infty$. Thus the Gel'fand space $\Phi_{\ell^1(\mathbb N,\omega)}$ is contained in $c_0(\mathbb N,\omega^{-1})$ which is a predual of $\ell^1(\mathbb N,\omega)$.

	\item Let $G$ be a discrete abelian group, let $\omega$ be an admissible weight on $G$, i.e., $\lim_{n\to \infty}\omega(s^n)^{\frac{1}{n}}=1\;(s \in G)$, and let  $\omega^{-1}$ vanish at infinity. Then the Banach algebra $\ell^1(G,\omega)$ has the predual $c_{0}(G, \omega^{-1})$. Note that $\Phi_{\ell^1(G,\omega)}=\{\varphi_\chi:\chi:G\to \mathbb T \text{ is a homomorphism}\}$, where $\varphi_\chi(f)=\sum_{s\in G}f(s)\chi(s)\;(f \in \ell^1(G,\omega))$. We have an identification of $\varphi_\chi\in \Phi_{\ell^1(G,\omega)}$ with $\chi\in c_0(G,\omega^{-1})$. The element $\chi$ will be in $c_0(G,\omega^{-1})$ as $\omega^{-1}$ vanishes at infinity.
	
	\item Every unital reflexive Banach algebra. In particular, any finite dimensional unital commutative Banach algebra.
\end{enumerate}

	\end{document}